\documentclass[12pt]{amsart}
\usepackage{amsfonts}
\usepackage{amsmath}
\usepackage{amssymb,latexsym}
\usepackage{verbatim}
\usepackage[usenames,dvipsnames]{xcolor}
\usepackage{enumitem}
\usepackage{wasysym}
\usepackage[all]{xy}
\usepackage{graphicx}
\usepackage{longtable}
\usepackage[left=2.41cm,top=3.5cm,bottom=4cm,right=2.41cm,a4paper]{geometry}
\usepackage[linkcolor=linkblue, citecolor=citegreen, ocgcolorlinks, bookmarksopen=True]{hyperref}
\usepackage{amsthm}
\usepackage[capitalize, noabbrev]{cleveref}

\usepackage{mathtools}
\usepackage{tikz}
\usepackage{comment}
\newtheorem{innercustomthm}{Theorem}
\newenvironment{customthm}[1]
  {\renewcommand\theinnercustomthm{#1}\innercustomthm}
  {\endinnercustomthm}

\def\dosubref#1#2:#3\relax{#1{#2}\label{#2:#3}}
\definecolor{linkblue}{RGB}{1,1,190}
\definecolor{citegreen}{RGB}{1,190,1}
\theoremstyle{plain}
\newtheorem{theorem}{Theorem}[section]
\newtheorem{corollary}[theorem]{Corollary}
\newtheorem{lemma}[theorem]{Lemma}
\newtheorem{proposition}[theorem]{Proposition}

\theoremstyle{definition}

\newtheorem{example}[theorem]{Example}

\numberwithin{equation}{section} 
\setlist[enumerate,1]{label=\textup{(\arabic*)},ref=\textup{(\arabic*)},leftmargin=2.5em}
\newlist{enumerateequiv}{enumerate}{1}  
\setlist[enumerateequiv,1]{label=\textup{(\alph*)},ref=\textup{(\alph*)},leftmargin=2.5em}

\DeclareMathOperator{\ann}{\rm ann}
\DeclareMathOperator{\diam}{\rm diam}

\begin{document}
\title{Domination and Total Domination Numbers in Zero-divisor Graphs of Commutative Rings}
\author[S.~Anderson]{Sarah Anderson}
\author[M.~Axtell]{Mike Axtell}
\author[B.~Kroschel]{Brenda Kroschel}
\author[J.~Stickles]{Joe Stickles}
\address{Department of Mathematics \\
St. Thomas University \\
St.~Paul, MN 55105}
\email{ande1298@stthomas.edu}
\address{Department of Mathematics \\
St. Thomas University \\
St.~Paul, MN 55105}
\email{mike.axtell@stthomas.edu}
\address{Department of Mathematics \\
St. Thomas University \\
St.~Paul, MN 55105}
\email{bkkroschel@stthomas.edu}
\address{School of Mathematics and Computational Sciences \\ Millikin University \\ Decatur, IL 62522}
\email{jstickles@millikin.edu}
\date{\today}
\begin{abstract}
Zero-divisor graphs of commutative rings are well-represented in the literature. In this paper, we consider dominating sets, total dominating sets, domination numbers and total domination numbers of zero-divisor graphs. We determine the domination and total domination numbers of zero-divisor graphs are equal for all zero-divisor graphs of commutative rings except for $\mathbb{Z}_2 \times D$ in which $D$ is a domain. In this case, $\gamma(\Gamma(\mathbb{Z}_2 \times D)) = 1$ and $\gamma_t(\Gamma(\mathbb{Z}_2 \times D)) = 2$.
\end{abstract}

\maketitle


\section{Introduction}
The concept of the graph of the zero-divisors of a commutative ring
was first introduced by Beck in \cite{BEC} when discussing the
coloring of a commutative ring. In his work all elements of the ring
were considered vertices of the graph. Since the seminal paper by
D.F. Anderson and Livingston \cite{AL1}, the standard is to regard
only nonzero zero-divisors, denoted $Z(R)^*$, as vertices of the graph, and we adhere
to this standard. The \emph{%
zero-divisor graph} of $R$, denoted $\Gamma (R),$ is the graph with $%
V(\Gamma (R))=Z(R)^{\ast },$ and for distinct $r,s\in Z(R)^{\ast }$, $r-s \in E(\Gamma(R))$ if and only if $rs=0$. Among other results, Anderson and Livingston proved that $\Gamma(R)$ is always connected and has diameter at most three (\cite[Theorem 2.3]{AL1}).  The discovery of strong graphical structure in zero-divisor graphs has inspired researchers to continue exploring how the graphical structure of the zero-divisor graph might reveal information about the algebraic structure in the ring, a desire necessitated by the frequent lack of closure under addition in the set of zero-divisors of a ring. For general surveys of $\Gamma(R)$, see \cite{AAS} and \cite{CWSS}. 

The main focus of this paper concerns the domination and total domination numbers of zero-divisor graphs. Our main result is the following theorem.

\begin{theorem}
\label{mainthm}
Let $R$ be a commutative ring. If $R$ is not isomorphic to $\mathbb{Z}_2 \times D$ where $D$ a domain, then $\gamma_t ( \Gamma (R))=\gamma( \Gamma (R))$.
\end{theorem}

In Section \ref{def}, an overview of definitions and known results is given. In Section \ref{girth4}, commutative rings whose zero-divisor graphs have a total domination number of $1$ or $2$ are investigated and connections to the girth of the zero-divisor graph are provided. In Section \ref{finitedominationnumber}, it is proven that if $R$ is not isomorphic to $\mathbb{Z}_2 \times D$ where $D$ is a domain and the total domination number of a zero-divisor graph of $R$ is finite, then the domination and the total domination numbers of the zero-divisor graph are equal. Finally, in Section \ref{mainresult}, Theorem \ref{mainthm} is proven.

\section{Definitions}
\label{def}
Throughout, by a \emph{ring} we mean a commutative ring, typically denoted by $R$, with a multiplicative identity 1. We use $Z(R)$ to denote the set of zero-divisors of $R$
and $Z(R)^{\ast }$ to denote the set of nonzero zero-divisors.  For the set of integers modulo $n$, we use the notation $\mathbb{Z}_n$.  There are two ideals of $R$ that will be of particular interest.  For $a \in R$, the \emph{annihilator} of $a$ is $\ann(a) = \{x \in R \mid ax = 0\}$.  The \emph{Jacobson radical} of a ring $R$, denoted $J(R)$, is the intersection of the maximal ideals of $R$.  A ring $R$ is said to be \emph{Artinian} if it satisfies the descending chain condition; i.e., there is no infinite descending sequence of ideals in $R$. It is well known that the Jacobson radical of an Artinian ring is nilpotent, that every element of an Artinian ring is either a unit or a zero-divisor, and that any Artinian ring is a direct product of fields and local rings.  In particular, all finite rings are Artinian.  A ring $R$ is said to be local if it is Noetherian with a unique maximal ideal.  It is also well known that Artinian rings are always Noetherian.  For a general algebra reference, see \cite{HUN}.

For any graph $G$, we denote the set of vertices of $G$ by $V(G)$ and the set of edges by $E(G)$.  We write $v - w$ when vertices $v$ and $w$ are \emph{adjacent}, or are incident on the same edge. If $v \in V(G)$ is such that $v - w \in E(G)$ for all $w \in V(G) \setminus \{v\}$, then $v$ is called a \emph{universal vertex}. By a \emph{path} between $v$ and $w$, we mean a sequence of vertices and edges $v - x_1 - x_2 - \cdots - x_n - w$, and  $G$ is \emph{connected} if there exists a path between any two distinct vertices.

The \emph{distance} between $v$ and $w$, denoted by $d(v,w)$, is the number of edges in a shortest path connecting 
$v$ and $w$ (note that $d(v,v) = 0$ and $d(v,w) = \infty$ if no such path exists). The 
\emph{diameter} of $G$ is $\diam(G) = \sup \{ d(v,w) \mid v,w \in V(G) \}$. In \cite[Theorem 2.3]{AL1}, it was shown that for any commutative ring $R$, $\Gamma(R)$ is connected with diameter $3$ or less. The \emph{girth} of a graph is the length of a shortest cycle contained in the graph. By \cite[Theorem 3]{AAS}, the girth of of a zero-divisor graph is $3$, $4$, or $\infty$. For a general graph theory reference,
see \cite{CHA}. 

For a graph $G$, a set $X \subseteq V(G)$ is a \emph{dominating set} of $G$ if for every $v \in V(G) \setminus X$	there exists $x \in X$ such that $x - v \in E(G)$.  The \emph{domination number} of $G$, denoted $\gamma(G)$, is $\gamma(G) = \min\{|X| \mid X  \text{ is a dominating set of } G\}$. Also for a graph $G$, we say $ X \subseteq V(G)$ is a \emph{total dominating set} of $G$ if for all $v \in V(G)$ there exists $x \in X$ such that $x - v$. We then define the \emph{total domination number} of $G$, denoted $\gamma_t ( G)$, to be the size of the minimum total dominating set.  Note that the difference between these two notions is that for a dominating set, we only require that every vertex not in the dominating set to be adjacent to a vertex in the dominating set, while for a total dominating set, even those vertices in the total dominating set must be adjacent to some vertex in the total dominating set.

Clearly, any total dominating set is also a dominating set, hence $\gamma ( \Gamma (R)) \leq \gamma_t ( \Gamma (R))$. It is not the case in general that dominating sets are total dominating sets. In fact, a minimum dominating set might not even be a total dominating set.  Consider $P_4$, the path graph on $4$ vertices. In Figure \ref{fig29}, the shaded dots in the left graph comprise a minimum dominating set of $P_4$, while the shaded dots in the right graph are a (minimum) total dominating set of $P_4$.

\begin{center}
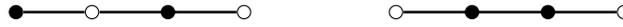
\begin{figure}[h]
\begin{tikzpicture}
\draw [line width=1.pt] (0.1,0)-- (0.9,0);
\draw [line width=1.pt] (1.1,0)-- (2,0);
\draw [line width=1.pt] (2,0)-- (2.9,0);
\draw [line width=1.pt] (5.1,0)-- (6,0);
\draw [line width=1.pt] (7,0)-- (6,0);
\draw [line width=1.pt] (7,0)-- (7.9,0);
\draw[fill=black] (0,0) circle (2.5pt);
\draw (1,0) circle (2.5pt);
\draw[fill=black] (2,0) circle (2.5pt);
\draw (3,0) circle (2.5pt);

\draw (5,0) circle (2.5pt);
\draw[fill=black] (6,0) circle (2.5pt);
\draw[fill=black] (7,0) circle (2.5pt);
\draw (8,0) circle (2.5pt);

\end{tikzpicture}
\caption{Minimum versus total dominating set in $P_4$}\label{fig29}

\end{figure}
\end{center}
\begin{example}
In Figure \ref{fig1}, we see that $\{(1,0),(0,1)\}$ is both a minimum dominating and a minimum total dominating set of $\Gamma(\mathbb{Z}_3 \times \mathbb{Z}_3)$, so $\gamma(\Gamma(\mathbb{Z}_3 \times \mathbb{Z}_3)) = \gamma_t(\Gamma(\mathbb{Z}_3 \times \mathbb{Z}_3)) = 2$.  In Figure \ref{fig2}, we see that $\{3\}$ is a minimum dominating set of $\Gamma(\mathbb{Z}_6)$, so $\gamma(\Gamma(\mathbb{Z}_6)=1$. However, there does not exist a total dominating set of size $1$, so $\gamma_t(\Gamma(\mathbb{Z}_6))=2$. The set $\{2,3\}$ is a total dominating set of $\Gamma(\mathbb{Z}_6)$.
\end{example}

\begin{center}
\begin{figure}[h]
\begin{tikzpicture}
\draw [line width=1pt] (0,0)-- (2,0);
\draw [line width=1pt] (0,0)-- (0,-2);
\draw [line width=1pt] (2,-2)-- (2,0);
\draw [line width=1pt] (0,-2)-- (2,-2);
\draw[fill=black] (0,0) circle (2.5pt);
\draw (-0.1,0.5) node {$(1,0)$};
\draw[fill=black]  (2,0) circle (2.5pt);
\draw(1.95,0.5) node {$(0,1)$};
\draw[fill=white](0,-2) circle (2.5pt);
\draw (-0.1,-2.5) node {$(0,2)$};
\draw[fill=white](2,-2) circle (2.5pt);
\draw (1.95,-2.5) node {$(2,0)$};
\end{tikzpicture}
\caption{$\Gamma(\mathbb{Z}_3 \times \mathbb{Z}_3)$} \label{fig1}
\end{figure}
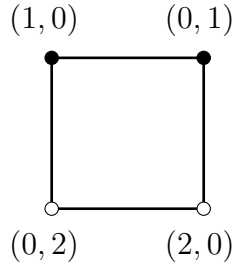
\end{center}

\begin{center}
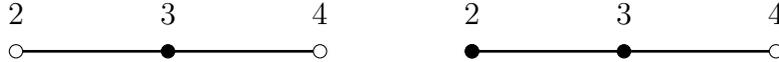
\begin{figure}[h]
\begin{tikzpicture}
\draw [line width=1pt] (0,0)-- (2,0);
\draw [line width=1pt] (4,0)-- (2,0);

\draw[fill=white] (0,0) circle (2.5pt);
\draw (0,0.5) node {$2$};
\draw[fill=black] (2,0) circle (2.5pt);
\draw (4,0.5) node {$4$};
\draw[fill=white] (4,0) circle (2.5pt);

\draw (2,0.5) node {$3$};

\draw [line width=1.pt] (6,0)-- (8,0);
\draw [line width=1.pt] (8,0)-- (10,0);

\draw (6,0.5) node {$2$};

\draw (6,0) circle (2.5pt);
\draw[fill=black](6,0) circle (2.5pt);
\draw (8,0.5) node {$3$};
\draw[fill=black] (8,0) circle (2.5pt);
\draw (10,0.5) node {$4$};
\draw[fill=white] (10,0) circle (2.5pt);

\end{tikzpicture}  

\caption{Dominating set versus total dominating set of $\Gamma (\mathbb{Z}_6)$.}\label{fig2}
\end{figure}
\end{center}

If $G$ is a simple graph (an undirected graph with no repeated edges and no looped vertices), then the definition of total dominating set requires that we have $\gamma_t(G) \geq 2$.  However, for zero-divisor graphs, we will have instances when a vertex, as an element of the ring, squares to 0.  Therefore, it will be beneficial to modify our zero-divisor graph definition to allow for self-adjacency.  If $x \in V(\Gamma(R))$, then we say $x$ is \emph{self-adjacent} (or is a \emph{neighbor to itself} or is \emph{looped}) if and only if $x^2 = 0$. Self-adjacency allows for the possibility of the total domination number of a graph equaling 1.

\begin{example}
\label{exz8}
In $\Gamma(\mathbb{Z}_8)$, the vertex set is $Z(R)^* = \{2,4,6\}$.  We see that the vertex $4$ is self-adjacent since $4^2 = 0$ in $\mathbb{Z}_8$. (We denote this self-adjacency with a loop on the vertex $4$.  See Figure \ref{figz8}.)  Further, both $2 - 4$ and $6 - 4 \in E(\Gamma(\mathbb{Z}_8))$. So, $\{4\}$ is a total dominating set of $\Gamma(\mathbb{Z}_8)$, and $\gamma_t(\Gamma(R)) = \gamma (\Gamma (R)) = 1$.
\end{example}

\begin{center}
\begin{figure}[h]
\begin{tikzpicture}
\draw [line width=1pt] (0,0)-- (2,0);
\draw [line width=1pt] (4,0)-- (2,0);
\draw[fill=white] (0,0) circle (2.5pt);
\draw (0,-0.5) node {$2$};
\draw[fill=white] (4,0) circle (2.5pt);
\draw (4,-0.5) node {$6$};
\draw[fill=black] (2,0) circle (2.5pt);
\draw (2,-0.5) node {$4$};
\draw (2,0.35) circle (10 pt);
\end{tikzpicture} 

\caption{Dominating and total domination set of $\Gamma (\mathbb{Z}_8)$}\label{figz8}
\end{figure}
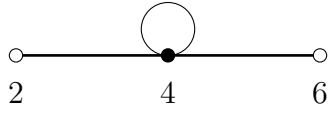
\end{center}


\section{zero-divisor graphs with girth 4 or infinity}
\label{girth4}
In this section, we start by finding zero-divisor graphs, other than $\Gamma(\mathbb{Z}_8)$, whose total domination number is $1$ by first recasting a well known result from \cite{AL1}.  Note that if $Z(R) = \ann(a)$ for some $a \in R$, we call $Z(R)$ an \textit{annihilator ideal}. 

\begin{theorem}\cite[Theorem 2.5]{AL1}\label{Result00}
Let $R$ be a commutative ring with identity. Then $\gamma ( \Gamma (R))=1$ if and only if $R \cong \mathbb{Z}_2 \times D$ where $D$ is a domain or $Z(R)$ is an annihilator ideal.
\end{theorem}

\begin{proof}
The condition that $\gamma ( \Gamma (R))=1$ is equivalent to $\Gamma (R)$ containing a universal vertex. The result now follows directly from \cite[Theorem 2.5]{AL1}.
\end{proof}

Requiring the total domination number to be $1$ provides very clear structural information about $Z(R)$.

\begin{theorem}\label{Result0}
Let $R$ be a commutative ring with identity. Then, $\gamma_t ( \Gamma (R))=1$ if and only if $Z(R)$ is an annihilator ideal.
\end{theorem}

\begin{proof}
The result follows directly from Theorem \ref{Result00} since for a domain $D$, $\gamma_t ( \Gamma (\mathbb{Z}_2 \times D))=2$.
\end{proof}

It is of interest to note that it is possible for $Z(R)$ to be an ideal with $\gamma_t ( \Gamma (R)) \geq 2$. Such an example is constructed in Section 3 of \cite{TRAM}. In this ring, the zero-divisors form an ideal, but the zero-divisor graph has diameter $3$; hence, its total domination number is at least $2$. However, if $Z(R)$ is an annihilator ideal, then it must be the case that $\gamma_t(\Gamma (R)) = \gamma(\Gamma (R)) = 1$. In the case, when $R \cong \mathbb{Z}_2 \times D$ where $D$ a domain, we note that $\gamma_t ( \Gamma (R))=2$, while $\gamma ( \Gamma (R))=1$. As we see in the next result, $R \cong \mathbb{Z}_2 \times D$ where $D$ a domain is the only case in which $\gamma_t ( \Gamma (R)) \neq \gamma( \Gamma (R))$ when  $\gamma_t(\Gamma (R)) = 2$.

\begin{proposition}
\label{totaldom2}
Let $R$ be a commutative ring. If $\gamma_t ( \Gamma (R)) = 2$ and $R$ is not isomorphic to $\mathbb{Z}_2 \times D$ where $D$ a domain, then $\gamma_t ( \Gamma (R))=\gamma( \Gamma (R))$.
\end{proposition}
\begin{proof}
Since $\gamma_t ( \Gamma (R)) = 2$ and $R$ is not isomorphic to $\mathbb{Z}_2 \times D$, by Theorem \ref{Result00} there is not a universal vertex. Hence, $\gamma ( \Gamma (R)) \neq 1$.
\end{proof}

We now begin to investigate if $\gamma_t( \Gamma(R) ) = \gamma( \Gamma(R) )$ when $\gamma_t( \Gamma(R)) \geq 3$, which is shown in Proposition \ref{corF} to occur only when the girth of the zero-divisor graph is $3$. The next result  provides information about the structure of total dominating sets of zero-divisor graphs.  

\begin{lemma}\label{PropC}
Let $R$ be a commutative ring. If $\Gamma (R)$ has a minimum total dominating set $D = \{ a_1, a_2, \ldots , a_n \}$, then $a_ia_j = 0$ whenever $i \neq j$.  In other words, $D$ induces a complete subgraph in $\Gamma (R)$.
\end{lemma}

\begin{proof}
Let $R$ be a commutative ring with $\gamma_t ( \Gamma (R)) = n$ and a minimum total dominating set $D = \{ a_1, a_2, \ldots , a_n \}$. For sake of contradiction, suppose, without loss of generality, $a_1a_2 \neq 0$. Since $D$ is a total dominating set, we must have either have $a_1^2 = 0$, or $a_1 a_i = 0$ for some $3 \leq i \leq n$. Consider $\bar{D} = \{ a_1a_2, a_3, \ldots , a_n \} = ( D \backslash \{a_1, a_2 \} ) \cup \{ a_1a_2 \} $. If $xa_1 = 0$ or $xa_2 = 0$, then $x(a_1a_2 ) =0$. It is now easy to verify that $\bar{D}$ is a total dominating set of $\Gamma (R)$ of size $n-1$, which is a contradiction. Thus, $a_ia_j = 0$ whenever $i \neq j$, and hence $D$ induces a complete subgraph in $\Gamma (R)$.
\end{proof}

 Using this structure of minimum total dominating sets in zero-divisor graphs, we show that if the girth of the zero-divisor graph is either 4 or infinity, then the total domination number is at most 2.

\begin{proposition} \label{corF}
Let $R$ be a commutative ring. If $\gamma_t ( \Gamma (R)) \geq 3$, then the girth of $\Gamma (R)$ is $3$.
\end{proposition}

\begin{proof}
By Lemma \ref{PropC}, if $\gamma_t ( \Gamma (R)) \geq 3$, then there exists a $3$-cycle consisting of $3$ distinct elements of a minimum total dominating set. Hence, the girth of $\Gamma (R)$ would be $3$.
\end{proof}

\begin{corollary}\label{cor1a}
Let $R$ be a commutative ring. If the girth of $\Gamma (R)$ is $4$ or $\infty$, then $\gamma_t ( \Gamma (R)) = \gamma ( \Gamma (R))$ or $R \cong \mathbb{Z}_2 \times D$ where $D$ is a domain. 
\end{corollary}

\begin{proof}
By Proposition \ref{corF}, we have $\gamma_t ( \Gamma (R)) = 1$ or $2$. If $\gamma_t ( \Gamma (R)) = 1$, then $\gamma_t ( \Gamma (R)) = \gamma ( \Gamma (R))$ always. If $\gamma_t ( \Gamma (R)) = 2$, then by Proposition \ref{totaldom2}, $\gamma_t ( \Gamma (R)) = \gamma ( \Gamma (R))$ unless $R$ is isomorphic to $\mathbb{Z}_2 \times D$ where $D$ a domain. 
\end{proof}

\begin{proposition}\label{PropMul} 
Let $R$ be a commutative ring. If the girth of $\Gamma (R)$ is $4$, then $\gamma_t ( \Gamma (R)) = \gamma ( \Gamma (R)) = 2$.
\end{proposition}

\begin{proof}
By \cite[Theorem 2.3]{MUL}, $R$ is either isomorphic to a subring of a product of two domains, or is isomorphic to $D \times S$, where $D$ is a domain and $S$ is either $\mathbb{Z}_4$ or $\mathbb{Z}_2 [X]/(X^2)$. Since the girth of $\Gamma (R)$ is $4$, we know $R$ is not isomorphic to $\mathbb{Z}_2 \times D$ where $D$ a domain. Thus, $\Gamma (R)$ is either complete bipartite or a subgraph of a complete bipartite graph, which contains at least two vertices in each disjoint vertex set. Hence, $\gamma_t ( \Gamma (R)) = \gamma ( \Gamma (R))=2$. 
\end{proof}

\begin{proposition}\label{PropSarah} 
Let $R$ be a commutative ring with girth of $\Gamma (R)$ equal to $\infty$. Then either $R \cong \mathbb{Z}_2 \times D$ where $D$ is a domain, or $\gamma_t ( \Gamma (R)) = \gamma ( \Gamma (R)))$.
\end{proposition}

\begin{proof}
 Suppose $R \ncong \mathbb{Z}_2 \times D$ where $D$ a domain, then $\gamma_t ( \Gamma (R)) \leq 2$ by Lemma \ref{PropC}. If $\gamma_t ( \Gamma (R)) = 1$, then $\gamma_t ( \Gamma (R)) = \gamma ( \Gamma (R))$. If $\gamma_t ( \Gamma (R)) = 2$, then by Proposition \ref{totaldom2}, $\gamma_t ( \Gamma (R)) = \gamma ( \Gamma (R))$.
\end{proof}

Using results from \cite{DS}, we can classify all possibles cases for domination and total domination numbers when the girth of $\Gamma(R)$ is $\infty$.

\begin{theorem}\label{girthinf}
Let $R$ be a commutative ring with girth of $\Gamma(R)$ equal to $\infty$. Then
\begin{enumerate}
\item $\gamma_t(\Gamma(R)) = 2$ and $\gamma(\Gamma(R) )=1$ when $R \cong \mathbb{Z}_2 \times D$ for any domain $D$.
\item $\gamma_t(\Gamma(R)) = \gamma(\Gamma(R))=2$ when $R \cong \mathbb{Z}_2 \times \mathbb{Z}_4$ or $R \cong \mathbb{Z}_2 \times \mathbb{Z}_2[X]/(X^2)$.
\item $\gamma_t(\Gamma(R)) = \gamma(\Gamma(R))=1$ otherwise.
\end{enumerate}
\end{theorem}

\begin{proof}
It is straightforward to check (1) and (2).  For (3), by \cite[Theorem 1.7]{DS} and \cite[Theorem 1.12]{DS}, the only other rings up to isomorphism with girth of $\Gamma(R)$ equal to $\infty$ are $\mathbb{Z}_4$, $\mathbb{Z}_8$, $\mathbb{Z}_9$, $\mathbb{Z}_2[X]/(X^2)$, $\mathbb{Z}_3[X]/(X^2)$, $\mathbb{Z}_2[X]/(X^3)$, and $\mathbb{Z}_4[X]/(X^2-2, 2X)$, as well as a ring in which there is exactly one non-zero nilpotent element $x$ such that $x^2=0$, the ideal $\{0,x\}$ is prime, and $\ann(x)=Z(R)$.  In each of these cases, it is easy to see that $\gamma_t(\Gamma(R)) = \gamma(\Gamma(R))=1$.
\end{proof}

\section{Finite domination number}
\label{finitedominationnumber}
Investigating zero-divisor graphs with small total domination number has allowed us to determine that $\gamma_t(\Gamma (R)) = \gamma ( \Gamma (R))$ for zero-divisor graphs with girth $4$ or infinity when $R$ is not isomorphic to $\mathbb{Z}_2 \times D$ for any domain $D$. By \cite[Theorem 3]{AAS}, the girth of a zero-divisor graph is $3$, $4$, or infinity. Thus, it is left to determine if $\gamma_t(\Gamma (R)) = \gamma ( \Gamma (R))$ for zero-divisor graphs whose girth is $3$ and total domination number is greater than or equal to $3$. We start this section by showing $\gamma_t(\Gamma (R)) = \gamma ( \Gamma (R))$ when $\gamma_t ( \Gamma (R)) = 3$. 


\begin{proposition} \label{propG}
Let $R$ be a commutative ring. If $\gamma_t ( \Gamma (R)) = 3$, then $\gamma ( \Gamma (R)) = 3$
\end{proposition}

\begin{proof}
Assume $\gamma_t ( \Gamma (R))=3$ with $D = \{ x_1 , x_2 , x_3 \}$ a minimum total dominating set. Since $\gamma ( \Gamma (R)) \leq \gamma_t ( \Gamma (R))$, we only need to rule out the cases when $\gamma ( \Gamma (R))=1$ and $\gamma ( \Gamma (R))= 2$. Suppose $\gamma ( \Gamma (R)) = 1$ and $\{ a \}$ is a minimum dominating set. Let $b \in Z(R)^* \backslash \{ a \}$.  Then $ab=0$, which means $\{ a, b \}$ is a total dominating set of $\Gamma (R)$, a contradiction. 

Now suppose $\gamma ( \Gamma (R)) = 2$ and $\{ a , b \}$ is a dominating set.  If $ab = 0$, then $\{ a , b \}$ is a total dominating set, a contradiction. Hence $ab \neq 0$. By Lemma \ref{PropC}, $x_1x_2 = x_1x_3=x_2x_3=0$. 
If, without loss of generality, $ax_1 = bx_1 = 0$, then $\{ ab, x_1 \}$ is a total dominating set, also a contradiction.

Finally, suppose that $a$ and $b$ do not have a common annihilator from $D$. Without loss of generality, we have $ax_1 = bx_2 = 0$ with $ax_2 \neq 0$ and $bx_1 \neq 0$. Consider $\bar{D} = \{ ax_2, bx_1\}$. It is straightforward to see that $\bar{D}$ is a total dominating set, a contradiction. Thus, $\gamma ( \Gamma (R)) = 3$.
\end{proof}

We can extend this result to all zero-divisor graphs with finite domination number whose total domination number is greater than $3$.

\begin{theorem}
\label{totaldom4}
Let $R$ be a commutative ring such that $4\leq \gamma _{t}\left( \Gamma \left( R\right) \right) <\infty .$ Then $\gamma \left( \Gamma \left( R\right) \right) =\gamma _{t}\left( \Gamma \left( R\right) \right) .$
\end{theorem}

\begin{proof}
We assume $\gamma \left( \Gamma \left( R\right) \right) <\gamma _{t}\left(
\Gamma \left( R\right) \right) $ and arrive at a contradiction. Among all
minimum dominating sets of $\Gamma \left( R\right) $, select a minimum
dominating set $D$ which maximizes the number of elements $z \in D$ such that $\ann(z) \cap D \neq \emptyset$. Note that since $D$
is not a total dominating set, there exists $x\in D$ such that $xy\neq 0$ for all $y\in D$; i.e., $ann(x) \cap D = \emptyset$.

Let $T$ be a minimum total dominating set of $\Gamma \left( R\right) ,$ so
there exists $c\in T$ such that $cx=0.$ Since no element of $D$ dominates $%
x, $ $c\notin D.$ For sake of contradiction, assume $T\setminus \left\{ c\right\} \subseteq D$. Since $\left\vert D\right\vert <\left\vert
T\right\vert ,$ it must the case that  $D=T \setminus \left\{ c\right\}.$ Hence, $x \in T$. Since $\left\vert
T\right\vert \geq 4,$ there exists $f\in T \setminus \left\{ x,c\right\} ,$
and since $T$ induces a complete subgraph of $\Gamma \left( R\right)$ by Lemma \ref{PropC}, $xf=0.$ However, $f\in D$, and $xy\neq 0$ for all $y\in D,$ is a contadiction. Therefore, $T \setminus \{c\}$ is not a subset of $D$, and
there must exist $d\in T \setminus \left\{ c\right\} $ such that $d\notin D.$

Now consider $x-d\in R.$ We chose $c$ so that $cx=0,$ and $cd=0$ since,
again, $T$ induces a complete subgraph of $\Gamma \left( R\right) .$
Therefore, $c\left( x-d\right) =0,$ and since $x\neq d,$ we have $x-d\in
Z\left( R\right) ^{\ast }.$ There are two cases to consider.

\textbf{Case 1}:\ Assume for every $y\in D$, $y\left( x-d\right)
\neq 0.$ Since $D$ is a dominating set, this implies $x-d\in D$ and $\left(
x-d\right) ^{2}\neq 0.$ If for every $z\in \ann\left( x\right) ^{\ast }$
there exists $f\in D\setminus \left\{ x\right\} $ such that $fz=0,$ then $%
\tilde{D}=\left( D \setminus \left\{ x\right\} \right) \cup \left\{
c\right\} $ is a minimum dominating set, and since $c\left( x-d\right) =0,$ we see that $\tilde{D}$ has more elements with annihilators having nontrivial intersection with $\tilde{D}$ than $D$ does, a contradiction. Thus, there must exist $s\in \ann\left( x\right) ^{\ast }$ such that 
$fs\neq 0$ for all $f\in D.$ Note that $s\neq x$ because $x^{2}\neq 0.$

If for every $z\in \ann\left( x-d\right) ^{\ast }$ there exists $f\in
D\setminus \left\{ x-d\right\} $ such that $fz=0,$ then $\tilde{D}=\left(
D\setminus \left\{ x-d\right\} \right) \cup \left\{ c\right\} $ is a
minimum dominating set, and since $cx=0,$  we see that $\tilde{D}$ has more elements with annihilators having nontrivial intersection with $\tilde{D}$ than $D$ does. Thus, there must exist $s\in \ann\left( x\right) ^{\ast }$ such that $fs\neq 0$ for all $f\in D.$
Thus, as before, there must exist $w\in \ann\left( x-d\right) ^{\ast }$ such
that $fw\neq 0$ for all $f\in D\setminus \left\{ x-d\right\} .$

By choice of $s$ and $w,$ we have $s\left( x-d\right) \neq 0$ and $wx\neq 0.$
Further, $\ann\left( x-d\right) \subseteq \ann\left( s\left( x-d\right)
\right) $ and $\ann\left( x\right) \subseteq \ann\left( wx\right) ,$ so $%
\tilde{D}=D\setminus \left\{ x,x-d\right\} \cup \left\{ s\left( x-d\right)
,wx\right\} $ is a dominating set of $\Gamma \left( R\right) $ with 
$|\tilde{D}|\leq \left\vert D\right\vert .$ With $D$
minimum, we must have $|\tilde{D}|=\left\vert D\right\vert $. Finally, since 
$\left( wx\right) \left( s\left( x-d\right) \right) =0,$ we see that $\tilde{D}$ has more elements with annihilators having nontrivial intersection with $\tilde{D}$ than $D$ does,
contradicting our choice of $D$.

\textbf{Case 2}: Assume that there exists $y\in D$ such that $y\left(
x-d\right) =0.$ There are three possibilities.

\begin{enumerate}
\item $yd=0$

\item $yd\neq 0$ and $y=x$ is the only element of $D$ that satisfies $%
y\left( x-d\right) =0.$

\item $yd\neq 0$ and there exists $y\in D\backslash \left\{ x\right\} $ such
that $y\left( x-d\right) =0.$
\end{enumerate}

For (1), if $yd=0,$ then we have $y\left( x-d\right) =yx=0,$ contradicting
our choice of $x.$

For (2), since $x\left( x-d\right) =0,$ we have $x^{2}=xd\neq 0.$ There
exists $s\in D\setminus \left\{ x\right\} $ such that $sd=0$. By our choice
of $x,$ we have $sx\neq 0.$ Further. we have $\left( sx\right)
x=sx^{2}=sxd=0.$ Now consider $\tilde{D}=\left( D\backslash \left\{
s,x\right\} \right) \cup \left\{ sx,x^{2}\right\} .$ We have shown $\left(
sx\right) x=0$ and $sx^{2}=0.$ Also, $\ann\left( x\right) \subseteq \ann%
\left( x^{2}\right) $ and $\ann\left( s\right) \subseteq \ann\left(
sx\right) .$ Thus, $\tilde{D}$ is a dominating set of $\Gamma \left(
R\right) $ with $|\tilde{D}|\leq |D|$; hence, $|\tilde{D}| = |D|$. However, since $\left( sx\right)
\left( x^{2}\right) =0$, we see that $\tilde{D}$ has more elements with annihilators having nontrivial intersection with $\tilde{D}$ than $D$ does,
contradicting our choice of $D$.

For (3), since $y\left( x-d\right) =0,$ we have $yx=yd\neq 0.$ First, assume 
$yd\in D.$ If $yd=x,$ then $\ann\left( y\right) \subseteq \ann\left(
x\right) $.  So, $\tilde{D}=\left( D\backslash \left\{
y\right\} \right) \cup \left\{ x-d\right\} $ is a dominating set of $\Gamma
\left( R\right) $ with $|\tilde{D}|\leq |D|$; hence, $|\tilde{D} = |D|$. Since $x\left( x-d\right)
=yd\left( y-d\right) =0,$ again $\tilde{D}$ has more elements with annihilators having nontrivial intersection with $\tilde{D}$ than $D$ does, contradicting our choice of $D.$ If $yd\in
D$ and $yd\neq x,$ then $\ann\left( x\right) \subseteq \ann\left( yd\right) .
$ So, $\tilde{D}=\left( D\backslash \left\{ x\right\} \right) \cup \left\{
c\right\} $ is a dominating set of $\Gamma \left( R\right) $ with $|\tilde{D}%
|=|D|$. Since $x\left( x-d\right) =yd\left( y-d\right) =0,$ the set $\tilde{D}$ has more elements with annihilators having nontrivial intersection with $\tilde{D}$ than $D$ does,
contradicting our choice of $D.$

Now assume for the remainder of the proof that $yd\notin D.$ Since $yx=yd,$
we have $\ann\left( x\right) \subseteq \ann\left( yd\right) $ and $\ann%
\left( y\right) \subseteq \ann\left( yd\right) .$ If there exists $w\in \ann%
\left( x\right) ^{\ast }\cap \ann\left( y\right) ^{\ast },$ then $\tilde{D}%
=\left( D\setminus \left\{ x,y\right\} \right) \cup \left\{ yd,w\right\} $
is a dominating set of $\Gamma \left( R\right) $ with $|\tilde{D}|\leq |D|.$
Since $D$ is minimum, we must have $|\tilde{D}|=|D|.$ Since $w\left(
yd\right) =0,$ the set $\tilde{D}$ has more elements with annihilators having nontrivial intersection with $\tilde{D}$ than $D$ does, contradicting our choice of $D.$ Thus, the only
common annihilator of $x$ and $y$ is 0.

We now show that neither $x\left( yd\right) \neq 0$ nor $y\left( yd\right)
\neq 0.$ If $x\left( yd\right) =0,$ then since $y\left( x-d\right) =0$, we
see that $\tilde{D}=\left( D\setminus \left\{ x,y\right\} \right) \cup
\left\{ yd,x-d\right\} $ is a dominating set of $\Gamma \left( R\right) $
with $|\tilde{D}|\leq |D|.$ Since $D$ is minimum, we must have $|\tilde{D}%
|=|D|.$ Since $yd\left( x-d\right) =0,$ the set $\tilde{D}$ has more elements with annihilators having nontrivial intersection with $\tilde{D}$ than $D$ does, contradicting our choice of 
$D.$ If $y\left( yd\right) =0,$ then $\tilde{D}=\left( D \setminus \left\{
x,y\right\} \right) \cup \left\{ yd,c\right\} $ is a dominating set of $%
\Gamma \left( R\right) $ with $|\tilde{D}|\leq |D|.$ Since $D$ is minimum,
we must have $|\tilde{D}|=|D|.$ Since $c\left( yd\right) =c\left( yx\right)
=0,$ the set $\tilde{D}$ has more elements with annihilators having nontrivial intersection with $\tilde{D}$ than $D$ does, again contradicting our choice of $D.$ Therefore, since $D$ is a
dominating set and $yd\notin D$, there exists $s\in D\backslash \left\{
x,y\right\} $ such that $s\left( yd\right) =0.$

If there exists $t\neq 0$ such that $tx=0$ and $st\neq 0,$ then we know  $
\ann\left( s\right) \subseteq \ann\left( st\right)$,  $\ann\left( x\right),
\ann\left( y\right) \subseteq \ann\left( yd\right)$, $\left( st\right) x=0$, $y\left( x-d\right) =0$, and $s\left( yd\right) =0$.  So, the set $\tilde{D}=\left( D\backslash \left\{ x,y,s\right\} \right) \cup \left\{
st,x-d,yd\right\} $ is a dominating set of $\Gamma \left( R\right) $ with $|%
\tilde{D}|\leq |D|.$ Since $D$ is minimum, we must have $|\tilde{D}|=|D|.$
Since $\left( st\right) \left( yd\right) =0=yd\left( x-d\right) ,$ set $\tilde{D}$ has more elements with annihilators having nontrivial intersection with $\tilde{D}$ than $D$ does, contradicting our choice of $D.$ 

Similarly, if there exists $t\neq 0$
such that $ty=0$ and $st\neq 0,$ then we know $\ann\left( s\right) \subseteq %
\ann\left( st\right) ,$ $\ann\left( x\right) ,\ann\left( y\right) \subseteq %
\ann\left( yd\right) ,$ $\left( st\right) y=0,$ $cx=0$. and $s\left(
yd\right) =0.$ So, $\tilde{D}=$ $\left( D\backslash \left\{ x,y,s\right\}
\right) \cup \left\{ st,c,yd\right\} $ is a dominating set of $\Gamma \left(
R\right) $ with $|\tilde{D}|\leq |D|.$ Since $D$ is minimum, we must have $|%
\tilde{D}|=|D|.$ Since $\left( st\right) \left( yd\right) =0=c\left(
yx\right) =c\left( yd\right) ,$ the set $\tilde{D}$ has more elements with annihilators having nontrivial intersection with $\tilde{D}$ than $D$ does, again contradicting our choice of $D.$
Therefore, we have $\ann\left( x\right) ,\ann\left( y\right) \subseteq \ann%
\left( s\right) .$

Since $\ann\left( y\right) \subseteq \ann\left( s\right) ,$ then like
before, if $yz=0$ for some $z\in D\backslash \left\{ y\right\} ,$ then $%
D\backslash \left\{ y\right\} $ is a dominating set of $\Gamma \left(
R\right) ,$ contradicting the minimality of $D.$ If $y^{2}=0,$ then $\ann%
\left( y\right) \subseteq \ann\left( s\right) $ implies $sy=0,$ and we have
an element in $D\backslash \left\{ y\right\} $ that annihilates $y.$ Thus,
we must have $yz\neq 0$ for all $z\in D.$ Finally, let $\tilde{D}=\left(
D\backslash \left\{ x,y\right\} \right) \cup \left\{ c,x-d\right\} .$ Then $%
s\in \tilde{D}$ and $\ann\left( x\right) ,\ann\left( y\right) \subseteq \ann%
\left( s\right) ,$ and since $c\in \ann\left( x\right) ,$ we have $cs=0.$
Also, $c\left( x-d\right) =0.$ Therefore,  $\tilde{D}$ is a dominating set
of $\Gamma \left( R\right) $ with $|\tilde{D}|\leq |D|.$ Since $D$ is
minimum, we must have $|\tilde{D}|=|D|.$ Since $\left( st\right) \left(
yd\right) =0=c\left( yx\right) =c\left( yd\right) ,$ set $\tilde{D}$ has more elements with annihilators having nontrivial intersection with $\tilde{D}$ than $D$ does, contradicting
our choice of $D,$ completing our proof.
\end{proof}

By Proposition \ref{totaldom2}, Proposition \ref{propG}, Theorem \ref{totaldom4}, and the fact that $\gamma_t ( \Gamma (R))=\gamma( \Gamma (R))$ if $\gamma_t ( \Gamma (R))= 1$, we have the following result.

\begin{theorem}
\label{finitetotaldom}
Let $R$ be a commutative ring. If $\gamma_t ( \Gamma (R)) < \infty$ and $R$ is not isomorphic to $\mathbb{Z}_2 \times D$ where $D$ a domain, then $\gamma_t ( \Gamma (R))=\gamma( \Gamma (R))$.
\end{theorem}

Since any finite graph will have a finite total domination number, the next result follows from Theorem \ref{finitetotaldom}.

\begin{corollary}
\label{finitering}
Let $R$ be a finite commutative ring. If $R$ is not isomorphic to $\mathbb{Z}_2 \times D$ where $D$ a domain, then $\gamma_t ( \Gamma (R))=\gamma( \Gamma (R))$.
\end{corollary}

\section{Main Result}
\label{mainresult}
It remains to consider if domination number and total domination number are equal for zero-divisor graphs whose total domination number is infinite. 

\begin{lemma}
\label{infinite}
Let $R$ be a commutative ring. If $\gamma_t ( \Gamma (R))$ is infinite, then $\gamma_t ( \Gamma (R))=\gamma( \Gamma (R))$.
\end{lemma}
\begin{proof}
If the total domination number of a graph is infinite, then the domination number of the graph is infinite. To see this, suppose that $D = \{ x_1 , x_2 , \dots , x_n \}$ is a minimum dominating set of $\Gamma (R) $. One way to create a total dominating set is to add an annihilator of $x_i$ to the set for $1 \leq i \leq n$, which shows that total dominating number is also finite. The initial statement of the proof now follows via the contrapositive.

Thus, the only case we need to consider is when $\gamma( \Gamma (R))$ is countable, but $\gamma_t( \Gamma (R))$ is uncountable. For sake of contradiction, assume $\Gamma(R)$ is such a graph. Let $D$ be a minimum dominating set of $\Gamma(R)$. Then $D$ is countable. Then for each element of $D$, pick one of its neighbors and put it in a set $N$. Since $D$ is countable, $N$ is countable. Consider $T = D \cup N$. Then $T$ is a total dominating set. Since $T$ is the union of two countable sets, $T$ is countable. Thus, since $T$ is a countable total dominating set, the total domination number must be countable, which is a contradiction. 
\end{proof}

By Theorem \ref{finitetotaldom} and Lemma \ref{infinite}, we now have the following.

\begin{customthm}{\ref{mainthm}}
Let $R$ be a commutative ring. If $R$ is not isomorphic to $\mathbb{Z}_2 \times D$ where $D$ a domain, then $\gamma_t ( \Gamma (R))=\gamma( \Gamma (R))$.
\end{customthm}


\begin{thebibliography}{99}
\bibitem[AAS]{AAS} Anderson, D.F., Axtell, M., and Stickles, J., Zero-divisor graphs in commutative rings, in {\it Commutative Algebra, Noetherian and Non-Noetherian Perspectives}, Springer-Verlag, New York (2011), 23-45.

\bibitem[AL1]{AL1} Anderson, D.F. and Livingston, P.S., \emph{The zero-divisor graph of a commutative ring}, J. Algebra ~\textbf{217} (1999), 434--447.


\bibitem[TRAM]{TRAM} Axtell, M., Stickles, J. and Trampbachls, W., \emph{Zero-divisor ideals and realizable zero-divisor graphs}, Involve ~\textbf{2} (2009), no. 1.

\bibitem[BEC]{BEC} Beck, I., Coloring of commutative rings, \emph{J. Algebra}
\textbf{116} (1988), 208-226.

\bibitem[CHA]{CHA} Chartland, G., \emph{Introductory Graph Theory}, Dover
Publications, Inc., New York, 1977.


\bibitem[CWSS]{CWSS} Coykendall, J., Sather-Wagstaff, S., Sheppardson, L., and Spiroff, S., On zero-divisor graphs, in {\it Progress in Commutative Algebra 2}, De Gruyter, 
Berlin (2012), 241-299.

\bibitem[DS]{DS} DeMeyer, F., Schneider, K., Automorphisms and zero-divisor graphs of commutative rings, \emph{Internat. J. Commutative Rings} \textbf{1} (3) (2001), 93-105.


\bibitem[HUN]{HUN} Hungerford, T., \emph{Algebra}, Springer, New York, 1974.


\bibitem[MUL]{MUL} Mulay, S.B., \emph{Rings having zero-divisor graphs of small diameter or large girth}, Bull. Austral. Math. Soc. ~\textbf{72} (2005), 481--490.














\end{thebibliography}
\end{document}